\documentclass[oneside,reqno]{amsart}

\usepackage[T1]{fontenc}
\usepackage{hyperref}
\usepackage{amsmath}
\usepackage{amssymb}
\usepackage{amsthm}
\usepackage{mathtools}
\usepackage{booktabs}
\usepackage{subcaption}
\usepackage{tikz}
\usepackage{nicefrac}
\usepackage{enumitem}
\usepackage{amsaddr}

\makeatletter
\newcommand\sw@temp{\phi}
\let\phi\varphi
\let\varphi\sw@temp
\makeatother

\makeatletter
\renewcommand\sw@temp{\epsilon}
\let\epsilon\varepsilon
\let\varepsilon\sw@temp
\makeatother

\let\emptyset\varnothing

\newcommand{\naturals}{\mathbb{N}}

\newcommand{\reals}{\mathbb{R}}

\newcommand{\rationals}{\mathbb{Q}}

\DeclarePairedDelimiterX{\set}[1]\lbrace\rbrace{\setaux #1||\endsetaux}
\def\setaux#1|#2|#3\endsetaux{\if\relax\detokenize{#2}\relax #1 \else #1
\;\delimsize\vert\; #2 \fi}

\newcommand{\bigO}{\mathcal O}

\newcommand\ecolnumf[1]{\ensuremath{\chi'_{f}(#1)}}
\newcommand\ecolnum[1]{\ensuremath{\chi'(#1)}}

\newcommand\mcf[1]{\ensuremath{\mathrm{mc}_{f}(#1)}}
\newcommand\mc[1]{\ensuremath{\mathrm{mc}(#1)}}

\newcommand\maxdeg[1]{\ensuremath{\Delta_\text{max}(#1)}}
\newcommand\mindeg[1]{\ensuremath{\Delta_\text{min}(#1)}}

\newcommand\pu[1]{\ensuremath{p_U^\mathrm{#1}}}
\newcommand\pr[1]{\ensuremath{p_R^\mathrm{#1}}}

\newtheorem{theorem}{Theorem}
\newtheorem{lemma}[theorem]{Lemma}
\newtheorem{corollary}[theorem]{Corollary}

\theoremstyle{definition}
\newtheorem*{example}{Example}

\begin{document}

\title{Fairness in Graph-Theoretical Optimization Problems}

\author[C. Hojny et al.]{Christopher Hojny, Frits Spieksma and Sten Wessel}
\address{Department of Mathematics and Computer Science,\\Eindhoven University of Technology, Eindhoven, The Netherlands}
\email{\{c.hojny,f.c.r.spieksma,s.wessel\}@tue.nl}

\begin{abstract}
There is arbitrariness in optimum solutions of graph-theoretic problems that can give rise to unfairness.
Incorporating fairness in such problems, however, can be done in multiple ways.
For instance, fairness can be defined on an individual level, for individual vertices or edges of a given graph, or on a group level.
In this work, we analyze in detail two individual-fairness measures that are based on finding a probability distribution over the set of solutions.
One measure guarantees uniform fairness, i.e., entities have equal chance of being part of the solution when sampling from this probability distribution.
The other measure maximizes the minimum probability for every entity of being selected in a solution.
In particular, we reveal that computing these individual-fairness measures is in fact equivalent to computing the fractional covering number and the fractional partitioning number of a hypergraph.
In addition, we show that for a general class of problems that we classify as independence systems, these two measures coincide.
We also analyze group fairness and how this can be combined with the individual-fairness measures.
Finally, we establish the computational complexity of determining group-fair solutions for matching.
\end{abstract}


\maketitle

\section{Introduction} \label{sec:introduction}
Traditionally, when confronted with an instance of an optimization problem,
the instance is considered solved when a provably optimum solution has been
found.
After all, what more can be wished for?
Actually, there is more.
Indeed, we should acknowledge that there can be a large amount of
arbitrariness in selecting an optimum solution, and that this arbitrariness
can be perceived as a source of unfairness.
The following example is inspired by Farnadi et al.~\cite{starnaud2023adaptation}.
Imagine one is a patient with end-stage renal disease, and that there is a donor who is willing to donate a kidney but who is incompatible with the patient.
Imagine further that you and your incompatible donor enter a kidney exchange program; we will refer to the pair entering the program as a node.
Entering this program means that at regular moments, runs of software are made, and the output of such a run consists of a set of exchanges between nodes~\cite{biro2019building,starnaud2023adaptation}.
A patient from a node involved in such an exchange receives a kidney, other patients do not.
Clearly, it is quite relevant whether one is included in such an exchange or not.
While the software may find an optimum solution (according to some criterion, say maximizing the number of exchanges), it can be the case that some node is not selected for an exchange while being part of \emph{some} optimum solution.
Selecting a fixed exchange plan will certainly be perceived as unfair by every patient who does not receive a kidney.
In such cases it can be hard to explain that the software simply selects an optimum solution which somehow favors some nodes at the expense of some other nodes.
We give other examples of this phenomenon in Section~\ref{sec:applications}.
In this contribution, we pursue this matter from a fairness point of view.

To overcome this issue, we analyze an approach that has been used by,
among others, Farnadi et al.~\cite{starnaud2023adaptation} and Flanigan et al.~\cite{flanigan2021fair}.
Instead of proposing a fixed solution, the idea is to consider a pool of solutions from
which a solution is sampled according to some probability distribution.
In this way, different entities of the optimization problem (e.g., patients in
the kidney exchange example) are contained in a selected solution with a certain probability.
To model fairness, the probability distribution is required to satisfy some
additional criteria.
For instance, one could enforce that the likelihood~$p$ for each individual
entity to be present in a sampled solution is the same.
Among all such distributions, one is then interested in finding a
distribution that maximizes~$p$.
This maximum~$p$ is called \emph{uniform individual
  fairness}, denoted by~$p_U$, and can be found by solving a suitable
linear program (LP), similar to the approach in~\cite{starnaud2023adaptation,flanigan2021fair} (see Section~\ref{sec:model}).
Alternatively, one could also try to maximize the minimum probability of an
entity to be present in a solution, the so-called \emph{Rawlsian justice},
named after the fairness principles introduced by
Rawls~\cite{rawls1971theory}, denoted by~$p_R$.

Although the aforementioned sampling approach is not new, it has mostly
been studied from a computational point of view.
In this article, we take a fresh look into this approach.
While existing results in the literature mostly considered the approach by Flanigan et al.~\cite{flanigan2021fair} for particular problems, we aim to understand the approach on a structural level.
Among others, we ask the following central questions.
What is the complexity of finding a fair probability distribution?
Is there a characterization of~$p_U$ and~$p_R$ in terms of graph-theoretic parameters?
By interpreting solutions of graph-theoretic problems as set systems, we show that these fairness measures are connected to known concepts from hypergraph theory.
For a large relevant class of problems these set systems are actually so-called independence systems.
Informally this means that if~$A$ is the set of all entities modeled in an  optimization problem and~$X \subseteq A$ is a solution of the optimization
problem, then also every subset of~$X$ is a solution.
In Section~\ref{sec:independence-systems} we provide a formal description.

Other than arbitrariness in selecting a solution, unfairness can also be introduced through systemic bias in the definition of the model or the description of the problem.
For instance, for the kidney exchange example, patients with certain blood types are more likely to be matched, simply because they are compatible with a larger group of donors.
To achieve fairness among different types of patients, one can ensure representation of protected groups in the solution that is selected.
This is referred to as \emph{group fairness}.
Not much existing work combines individual fairness with group fairness, while in this work we show how both fairness approaches can be combined in a single framework.

Our contributions, answering the central questions above, can be summarized as follows:
\begin{itemize}
    \item We show how the individual-fairness measures $p_U$ and $p_R$ are connected to known concepts in fractional hypergraph theory.
    \item We show that for independence systems, $p_U = p_R$.
    \item We show that when the underlying optimization problem can be solved in polynomial time, then finding $p_U$ and $p_R$ can also be done in polynomial time, as well finding as the probability distributions achieving $p_U$ and $p_R$.
    \item We show how group fairness constraints can be applied to graph-theoretic problems, and how it can be combined with individual fairness.
    \item We present complexity results for group-fair matching problems.
\end{itemize}

We organize the paper as follows.
In the remainder of Section~\ref{sec:introduction}, we introduce notation and terminology (Section~\ref{subsec:notation}), and we discuss a number of problems and applications that motivate fairness in graph-theoretical optimization.
We then give an overview of related work in Section~\ref{subsec:related-work}.
In Section~\ref{sec:model}, we describe the general framework of modeling individual and group fairness in these type of problems, show the relation of the two individual fairness measures with hypergraph theory (Section~\ref{sec:hypergraph}), and analyze the complexity of computing the fairness measures (Section~\ref{sec:complexity}).
Then, in Section~\ref{sec:independence-systems} we apply the framework to a general type of problems we refer to as independence systems, where we show a main result that~$p_U = p_R$.
This is then applied to concrete graph-theoretical problems in Section~\ref{subsec:fair-matching:edges} and~\ref{subsec:independent-set}, followed by Section~\ref{sec:matching-vertex} where we discuss an additional application that is not an independence system.
In Section~\ref{sec:gf} we formalize group fairness in a general setting for graph-theoretical problems, and analyze the complexity of finding group-fair matchings in a graph.
We give an outlook for further directions of research in Section~\ref{sec:conclusion}.

\subsection{Notation and Terminology} \label{subsec:notation}
In this section, we introduce notation and terminology that we will use throughout this work.
Let~$G = (V, E)$ be a graph with finite vertex set~$V$ and edge set~$E \subseteq \binom{V}{2}$.
Two vertices are \emph{adjacent} if there is an edge between them.
For a vertex~$v \in V$, let~$N(v) \coloneqq \set{u : \set{u, v} \in E}$ denote the set of \emph{neighbors} of~$v$.
A vertex is \emph{isolated} when it has no neighbors.
Let~$\delta(v) \coloneqq \set{\set{u, v} : u \in N(v)}$ denote the set of edges \emph{incident} to~$v$.
The \emph{degree} of a vertex~$v$ is the number of incident edges, i.e., $|\delta(v)|$.
The minimum and maximum degree of the graph are denoted by~$\mindeg G$ and~$\maxdeg G$, respectively.
For a given subset~$S \subseteq V$, let~$N(S) \coloneqq \set{v : \set{u, v} \in E,\ u \in S,\ v \not\in S}$ denote the \emph{(open) neighborhood} of the set~$S$.

A \emph{subgraph} of $(V, E)$ is a graph~$G'=(V', E')$ where~$V' \subseteq V$ and~$E' \subseteq E$ and~$V' \supseteq \bigcup_{e \in E'} e$.
A subgraph is a \emph{spanning subgraph} when~$V' = V$.
The subgraph induced by the vertex set~$S$, denoted by $G[S]$, is the graph on vertex set~$S$ and exactly the edges from~$E$ that have both endpoints in~$S$.

We call a graph \emph{regular} or \emph{regular with degree~$d$} when every vertex has the same degree~$d$.
A graph is \emph{connected} when for every pair of distinct vertices~$u,v \in V$ there is a sequence of adjacent vertices, starting at~$u$ and ending in~$v$.
A \emph{(connected) component} of a graph is a connected subgraph that is maximal, in the sense that it is not part of a larger connected subgraph.
A sequence of adjacent vertices is a \emph{path} if it does not visit a vertex more than once.
A \emph{cycle} is a sequence of adjacent vertices that starts and ends at the same vertex and does not visit a vertex more than once.
A graph is \emph{Hamiltonian} if it contains an Hamiltonian cycle, a cycle that visits every vertex exactly once.

The \emph{complete graph} on~$n$~vertices, denoted by~$K_n$, is the graph on~$n$ vertices where there is an edge between every pair of vertices.
A graph is a \emph{cycle graph} when it is connected and regular with degree~$2$.

\subsection{Problems and Applications}
\label{sec:applications}
We describe here three problems and applications motivating the notion of fairness in graph-theoretical optimization.

First, the application of the kidney exchange program, as introduced above, can be modeled as a graph where vertices are patient-donor pairs and the edges represent compatibility between the donors with patients of the two pairs, see for a general overview Biro et al.~\cite{biro2019building}.
Finding a solution corresponds to the the well-known matching problem in this compatibility graph.
Given a graph $G=(V,E)$, a \emph{matching} is a subset of edges of the graph, such that no two
edges in the matching have a vertex in common.
A matching is \emph{maximum} if there is no matching with higher cardinality.
A matching is \emph{perfect} if it has cardinality~$\frac{|V|}{2}$, i.e., the
matching covers all vertices of the graph.
The \emph{matching number}~$\nu(G)$ is the maximum cardinality of any matching in the graph~$G$.
A vertex of the graph is \emph{covered} by a matching if it is an endpoint of one of the matching's edges.

We are interested in generating a matching in a given graph~$G$, with the property that our solution is \emph{fair} according to the following principle: every vertex in the graph has a priori an equal probability of being covered by the matching generated by our procedure.
We call this property \emph{uniform} or \emph{individual fairness}.
We achieve this by computing a probability distribution over the set of all matchings of the graph, such that sampling from this probability distribution indeed ensures that every vertex has equal probability, say~$p$, of being covered by the sampled matching.
We call this problem \emph{fair matching for vertices}.

Second, consider again the matching problem for some given graph~$G=(V,E)$. Now, however, we aim to establish a probability distribution over the matchings such that each \emph{edge} in $E$ has an equal probability of being present in a matching sampled from this probability distribution. Thus, in this setting, our solution is fair if each edge in~$E$ has equal chance of being present in a selected matching. We call this problem \emph{fair matching for edges}.

As a different application, consider a wireless (ad hoc) network, where devices (or the nodes of the network) transmit
messages to each other via radio signals~\cite{liu2009clique}.
A device has a certain transmission power, which dictates the range of how
``far'' it can broadcast a message.
When multiple devices that are in each other's range broadcast simultaneously, the
transmission becomes jammed.
In the corresponding \emph{conflict graph}, nodes correspond to devices, and two nodes are adjacent when a simultaneous
broadcast is not possible.
A strategy to avoid simultaneous broadcasting by adjacent nodes is to design a schedule where time\-slots
are assigned to nodes in such a way that they can safely transmit.
The nodes assigned per timeslot are so-called \emph{independent sets} in the conflict graph.
Next, a probability distribution over the independent sets can be seen as such a
schedule, where it can be important that every node gets the same
transmission time (to, e.g., balance energy consumption between the nodes), see, among others, \cite{liu2009clique}, \cite{fineman2014fair}, and \cite{huaizhou2013fairness} for a general overview of fairness in wireless (ad-hoc) networks.
Thus, given a graph~$G=(V,E)$ the goal is to select a maximum cardinality subset of the vertices such that each pair  of selected vertices is not connected by an edge from~$E$. The goal is to find a probability distribution over the independent sets such that each vertex has equal probability to be in a selected independent set. We call this problem \emph{fair independent set}.

Another application is centered around combinatorial auctions, see Cramton et al.~\cite{cramtonetalbook} for an overview.
Informally stated, there is a number of bidders as well as a set of items.
A bidder has a valuation, either implicitly or explicitly, for every subset of items (called a bundle).
Assuming these valuations are known, it is the auctioneer's task to find a maximum-valued assignment of bundles to bidders such that each item is part of at most one bundle.
Combinatorial auctions are well-known both in scientific literature, as well as in practice (spectrum, housing, etc.) and may involve huge amounts of money~\cite{cramtonetalbook}.
When confronted with the existence of multiple maximum-valued assignments where in one solution a bidder may receive an empty bundle, and in another one a nonempty bundle, legal repercussions are known to happen, see Goossens et al.~\cite{goossens2014auctions}.
One way of dealing with this phenomenon is randomization over (optimum) solutions, which we call \emph{procedural} or \emph{ex-ante} fairness.

In this work, we aim to apply randomization in such a way that we achieve individual ex-ante uniform fairness or Rawlsian justice.
At the same time, we achieve group fairness by only considering solutions that satisfy group fairness constraints, which is a form of \emph{ex-post} group fairness.

\section{Related Work} \label{subsec:related-work}
The design and analysis of fair algorithms in decision making has recently gained increasing attention in the research community.
The concept of \emph{individual fairness}, which was first introduced by Dwork et al.~\cite{Dwork2012} ensures that ``similar'' individuals should be treated similarly.
For classification problems, this is known as \emph{fairness through awareness}~\cite{Dwork2012}.
In our work, we consider every individual to be similar to every other individual, i.e., we make no distinction between any individual as we do not assume any distinguishing data or properties to be defined for individuals.
Literature on individual fairness in graph optimization includes
Garc\'ia-Soriano and Bonchi~\cite{GarciaSoriano2020a} who present a combinatorial algorithm to find a probability distribution over the matchings, in bipartite graphs, that achieves individual fairness with respect to a fairness measure that lexicographically maximizes the probabilities of vertices being included in a one-sided maximum matching.
They furthermore show theoretical properties of this measure for general problems where the solutions form a matroid.
Garc\'ia-Soriano and Bonchi~\cite{GarciaSoriano2021} extend similar results to ranking under group fairness constraints, and state a complexity result for general combinatorial problems for this fairness measure.
In our work, we extend this complexity result for the Rawlsian justice and uniform fairness measures in the general class of graph optimization problems, and compare Rawlsian justice with uniform fairness.
We show theoretical results for the broad class of independence systems, which generalize matroids.

Other work concerning individual fairness in a combinatorial setting include Rawlsian justice for online bipartite matching (Esmaeili et al.~\cite{Esmaeili2023}), selecting citizen assemblies (Flanigan et al.~\cite{flanigan2021fair} and Flanigan et al.~\cite{Flanigan2021}), and kidney exchange problems (Farnadi et al.~\cite{starnaud2023adaptation}).
A related approach is to address fairness while restricting to only optimal solutions.
In this case, there is no tradeoff between optimality and fairness, but the notion of individual fairness is weaker.
This has recently been applied for general integer programming problems by Demeulemeester et al.~\cite{Demeulemeester2023}. 

A large part of the literature considers \emph{group fairness}, where fairness constraints enforce that certain groups are protected from discrimination against a sensitive property such as race, gender, age, or income.
See for example the popular group fairness measures \emph{demographic parity} in Dwork et al.~\cite{Dwork2012}, and Calders and Verwer~\cite{Calders2010}, \emph{equal opportunity} in Moritz et al.~\cite{moritz2016}, and \emph{equalized odds} in Mortiz et al.~\cite{moritz2016}.
Much of the group fairness literature is applied to machine learning and classification, see for example the recent surveys by Mehrabi et al.~\cite{Mehrabi2021}, and Pessach and Shmueli~\cite{Pessach2022}.
Relevant literature on group fairness in combinatorial problems include Chierichetti et al.~\cite{chierichetti19a}, who give an approximation algorithm for group-fair solutions in matroids and matroid intersection.
Group fairness is also applied to shortest-path problems, where in a vertex-colored graph the number of colors in the selected path must be balanced (Bentert et al.~\cite{BENTERT2022149}).
Bandyapadhyay et al.~\cite{faircovhit} consider group fairness in geometric covering and hitting problems, and present an approximation algorithm for vertex cover and an exact polynomial-time algorithm for edge cover under group fairness constraints.

Combining individual fairness with group fairness is not well studied in the existing literature.
Garc\'ia-Soriano and Bonchi~\cite{GarciaSoriano2021} combine individual fairness with group fairness for the ranking problem, and provide some computational results.

\section{A Generic Framework for Modeling Fairness in Graphs} \label{sec:model}
In this section, we describe a generic framework for modeling fairness in combinatorial optimization problems, motivated by the three three settings that relate to the applications described earlier.
There is a common generalization of these three settings for which we formulate a model for finding a probability distribution
over the set of feasible solutions of a graph-theoretical optimization problem.
This model resembles fairness models for concrete problems that exist in the literature~\cite{starnaud2023adaptation,flanigan2021fair}.
Let~$A$ be a given \emph{ground set}, and let~$M$ be a family of
subsets of the ground set~$A$.
We will assume throughout that~$\emptyset \in M$ and that every element~$a \in A$ is contained in some set~$m \in M$.
We call the tuple~$(A, M)$ a \emph{set system}.

\begin{example}
    The fair matching for vertices problem arises when the ground set $A$ coincides with the vertex set $V$, and when $M$ coincides with the collection of subsets of vertices that are
    covered by matchings in~$G$.
    Further, fair matching for edges arises when the ground set~$A$ coincides with the edge set $E$, and when $M$ coincides with the collection of subsets of edges that are matchings in~$G$.
    Finally, when the ground set $A$ coincides with the vertex set $V$, and $M$ coincides with the collection of vertex sets that are independent sets, we arrive at the problem of finding fair independent sets in~$G$.
\end{example}

To achieve individual fairness, we want to find a probability distribution~$\set{x_m}_{m \in M}$ over the subsets in~$M$,
such that if we sample according to~$x$, every element~$a \in A$ has
equal probability~$p$ to be in the sampled subset.
Our aim is to maximize this probability, which we will call~$p_U$ (for \emph{uniform} probability).
For a given element~$a \in A$, let~$M_a \subseteq M$ denote the
collection of subsets from~$M$ that contain the element~$a$.
We model the problem of maximizing~$p_U$ with the following LP\@.
\begin{subequations}
    \label{eq:lp:pu}
    \begin{alignat}{9}
        p_U =\quad &&\text{maximize}\quad && p &&& \\
        &&\text{subject to}\quad && \sum_{m \in M_a} x_m &= p
        &&\qquad \forall {a \in A}, && \label{cons:pu:uniform} \\
        &&&& \sum_{m \in M}x_m &= 1, && && \label{cons:pu:probdist1}\\
        &&&& x_m &\ge 0 &&\qquad\forall {m \in M}, \label{cons:pu:probdist2}\\
        &&&& p &\in \reals. &&
    \end{alignat}
\end{subequations}
Constraints~\eqref{cons:pu:uniform} ensure that each element~$a \in A$ is selected with uniform probability, while Constraint~\eqref{cons:pu:probdist1} and Constraints~\eqref{cons:pu:probdist2} ensure that~$x$ is a probability distribution over~$M$.

We also consider a variant of this problem, where our notion of fairness
is relaxed to maximizing the minimum probability~$p_R$ that a ground set
element is selected.
This is also referred to as Rawlsian justice, and can be modeled with
the following LP\@.
\begin{subequations}
    \label{eq:lp:pr}
    \begin{alignat}{7}
        p_R =\quad &&\text{maximize}\quad && p &&& \\
        &&\text{subject to}\quad && \sum_{m \in M_a} x_m &\ge p
        &&\qquad \forall {a \in A}, \\
        &&&& \sum_{m \in M}x_m &= 1, &&\label{eq:lp:pr:probdist}\\
        &&&& x_m &\ge 0 &&\qquad\forall {m \in M},\\
        &&&& p &\in \reals. &&
    \end{alignat}
\end{subequations}
In this problem, it is guaranteed that~$p_R > 0$ (due to our assumption that every vertex is covered by at least one subset in~$M$).
In general, this is not the case for~$p_U$, which might equal zero, as we will discuss below in detail.

To include group fairness in this framework, we can ensure representation of groups in any solution in the support of the probability distribution, which can be done by restricting~$M$ to only solutions that meet the representation criteria.
We discuss in more detail how this can be done in Section~\ref{sec:gf}.
For the remainder of Section~\ref{sec:model} and Section~\ref{sec:if}, we first focus on individual fairness only.

\subsection{Relation to Fractional Hypergraph Theory} \label{sec:hypergraph}
In this section, we describe how the fairness measures $p_U$ and $p_R$ are related to concepts developed in fractional hypergraph theory.
We refer to Scheinerman and Ullman~\cite{scheinerman2011fractional}, and F\"uredi~\cite{furedi1988matchings} for an overview of fractional graph and hypergraph theory.
The main idea is to see the set system~$(A, M)$ as a \emph{hypergraph}~$\mathcal H$
with vertex set~$A$ and hyperedges~$M$.
A \emph{partition} of a hypergraph~$\mathcal H$ is a collection of disjoint hyperedges
such that every vertex is incident to exactly one hyperedge.
The \emph{partitioning number}~$k^{=}(\mathcal H)$ is the minimum
cardinality of a partition, if it exists, see~\cite{scheinerman2011fractional}.
If the hypergraph has no partition, we use the convention to define~$k^{=}(\mathcal H) = \infty$.
A \emph{fractional partition} of a hypergraph is an assignment of
non-negative weights~$w \in \reals_+^{M}$ to the hyperedges such that
for every vertex the sum of weights of the incident hyperedges is exactly~$1$.
The \emph{fractional partitioning number}~$k^{=}_f(\mathcal H)$ is then the minimum weight of a fractional partition.
The fractional partitioning number can be formulated as the following LP\@.
\begin{subequations}
    \label{eq:lp:kfh}
    \begin{alignat}{7}
        k^{=}_f(\mathcal H) =\quad &&\text{minimize}\quad && \sum_{m \in M} w_m &&& \\
        &&\text{subject to}\quad && \sum_{m \in M_a} w_m &= 1 \label{cons:kfh:eq}
        &&\qquad \forall {a \in A}, \\
        &&&& w_m &\ge 0 &&\qquad\forall {m \in M}.
    \end{alignat}
\end{subequations}
Note that indeed when a fractional partition does not exist, the program is infeasible.

Determining~$p_U$ can then be alternatively formulated as determining the fractional partitioning number, because of the following relation.
\begin{lemma} \label{lemma:pu-frac-hypergraph-part}
    Let~$\mathcal H = (A, M)$ be a hypergraph.
    Then it holds that~$p_U = \frac{1}{k^{=}_f(\mathcal H)}$ if~$k^{=}_f(\mathcal H)$ is finite, and~$p_U = 0$ otherwise.
\end{lemma}
\begin{proof}
    We first handle the case where~$k^{=}_f(\mathcal H) = \infty$.
    Then, the LP~\eqref{eq:lp:kfh} is infeasible.
    By multiplying the constraints~\eqref{cons:kfh:eq} with a given positive constant~$p$, it follows that there is no~$w \in \reals_+^M$ such that~$\sum_{m \in M_a} w_m = p$ for all~$a \in A$.
    Therefore, the only feasible solution for~\eqref{eq:lp:pu} is~$x_\emptyset = 1$ and~$x_m = 0$ for all~$m \in M\setminus\set{\emptyset}$, yielding that~$p_U = 0$.

    In case~$k^{=}_f(\mathcal H)$ is finite,
    let~$x^*$ be an optimal solution for~\eqref{eq:lp:pu}.
    Define~$w_m \coloneqq \frac{x^*_m}{p_U}$, for every~$m \in M$.
    Note that~$w$ is feasible for \eqref{eq:lp:kfh}, and
    moreover Constraint~\eqref{cons:pu:probdist1} yields that~$w$
    corresponds to an objective value of~$\frac{1}{p_U}$.
    Hence, $k^{=}_f(\mathcal H) \le \frac{1}{p_U}$.

    Further, let $w^*$ be an optimal solution
    for~\eqref{eq:lp:kfh}, and define~$x_m
    \coloneqq \frac{w_m^*}{k^{=}_f(\mathcal H)}$.
    This yields a solution for \eqref{eq:lp:pu} with objective value equal
    to~$\frac{1}{k^{=}_f(\mathcal H)}$.
    Hence, $p_U \ge \frac{1}{k^{=}_f(\mathcal H)}$, and the result follows.
\end{proof}

One can also define a \emph{cover} of a hypergraph as a
selection of hyperedges such that every vertex is in at least one hyperedge.
The \emph{covering number}~$k^{\ge}(\mathcal H)$ is the minimum
cardinality of a cover.
A \emph{fractional cover} of a hypergraph is an assignment of
nonnegative weights~$w \in \reals_+^{M}$ to the hyperedges such that
for every vertex the sum of weights of the incident hyperedges is at
least~$1$.
The \emph{fractional covering number}~$k^{\ge}_f(\mathcal H)$ is then the minimum weight of a fractional cover.
This is the fractional hypergraph invariant that corresponds to~$p_R$, as stated in the following lemma.
Notice that~$p_R > 0$ and also~$k^{\ge}_f(\mathcal H)$ is finite, as a covering always exists under our assumption that every~$a \in A$ is in at least one~$m \in M$.
The completion of the argument is analogous to the proof of Lemma~\ref{lemma:pu-frac-hypergraph-part}.
\begin{lemma} \label{lemma:pr-frac-hypergraph-cov}
    Let~$\mathcal H = (A, M)$ be a hypergraph.
    Then it holds that~$p_R = \frac{1}{k^{\ge}_f(\mathcal H)}$.
\end{lemma}
Concluding, finding $p_U$ and $p_R$ for a set system $(A,M)$ is nothing else but finding the fractional partitioning number~$k^{=}_f(\mathcal H)$ and the fractional covering number~$k^{\ge}_f(\mathcal H)$ corresponding to the hypergraph $\mathcal H = (A, M)$.

\subsection{Computational Complexity} \label{sec:complexity}

Due to the above discussion, a fair probability distribution with respect to~$p_U$
and~$p_R$ can be found by solving the linear programs~\eqref{eq:lp:pu}
and~\eqref{eq:lp:pr}, respectively.
Linear programs are solved routinely in practice by the simplex
algorithm~\cite{dantzig1951maximization}, and algorithms such as the ellipsoid
method~\cite{khachiyan1979polynomial} and interior point
methods~\cite{karmarkar1984new} have a provable polynomial running time.
Note, however, that this does not imply that the LPs~\eqref{eq:lp:pu}
and~\eqref{eq:lp:pr} can be solved in~$\bigO(\text{poly}(|A|))$ time,
as the running time of the ellipsoid and interior point methods depend
polynomially on the number of variables~$|M|$, which can depend
exponentially on~$|A|$.

This observation is only an obstacle at first glance though.
By using the concept of LP duality, we can observe that~$p_U$ and~$p_R$
can be found in polynomial time whenever the corresponding optimization
problem over~$M$ can be solved in polynomial time.
The argument is similar to the analysis of Garc\'ia-Soriano and Bonchi~\cite{GarciaSoriano2021}, who consider a lexicographic fairness measure.
\begin{theorem} \label{thm:complexity}
  Let~$A$ be a finite set and let~$M$ be a collection of subsets of~$A$
  such that~$\emptyset \in M$.
  If, for every~$c \in \rationals^A$, a set~$m \in M$ that
  minimizes~$\sum_{a \in m} c_a$ can be found in time polynomial in~$|A|$,
  then~$p_U$ and~$p_R$ can be found in~$\bigO(\text{poly}(|A|))$ time.
\end{theorem}
\begin{proof}
  We only provide the arguments for~$p_U$ as the argumentation for~$p_R$ is
  analogous.

  The dual of~\eqref{eq:lp:pu} is given by
  \begin{subequations} \label{eq:lp:pu:dual}
  \begin{alignat}{7}
      &&\text{minimize}\quad && \beta &&& \\
      &&\text{subject to}\quad && \sum_{a \in m} \alpha_a + \beta &\ge 0
      &&\qquad \forall {m \in M}, \label{eq:ip:pdual:m} \\
      &&&& \sum_{a \in A}\alpha_a &= -1, &&\\
      &&&& \alpha_a &\in \reals &&\qquad\forall {a \in A},\\
      &&&& \beta &\in \reals. &&
  \end{alignat}
  \end{subequations}
  Note that variables~$\alpha$ and~$\beta$ correspond to the Constraints~\eqref{cons:pu:uniform} and~\eqref{cons:pu:probdist1}, respectively.
  The strong duality theorem of linear programming ensures that the dual
  problem has an optimal solution if and only if~\eqref{eq:lp:pu} does.
  Moreover, both the optimal primal and dual objective value coincide.
  Thus, since~$\emptyset \in M$ (by assumption), LP~\eqref{eq:lp:pu} is feasible.
  In particular, as~\eqref{eq:lp:pu} is obviously bounded, it has an
  optimal solution.
  Consequently, the dual problem has the same optimal objective value and
  we can find the value of~$p_U$ by solving the dual problem.

  Due to the seminal result by Gr\"otschel, Lov\'asz, and
  Schrijver~\cite{grotschel1981ellipsoid}, an optimal solution of the dual can be found in
  polynomial time provided the separation problem for its constraints can be
  solved in polynomial time.
  The separation problem receives as input a vector~$(\alpha^*, \beta^*) \in
  \rationals^A \times \rationals$
  and asks whether there exists a violated dual constraint.
  For the single equality constraint, this question can certainly be
  answered in~$\bigO(\text{poly}(|A|))$ time.
  For the inequality constraints, it is sufficient to find an inequality
  whose left-hand side evaluated in~$(\alpha^*, \beta^*)$ is minimal.
  This problem reduces to finding a set~$m \in M$ such that~$\sum_{a \in m}
  \alpha^*_a$ is minimized.
  By assumption, this problem can be solved in polynomial time, which
  concludes the proof.
\end{proof}

Examples where minimizing~$\sum_{a \in m} c_a$ can be done in polynomial time are matching, shortest paths, spanning trees, and many other combinatorial problems.
By Theorem~\ref{thm:complexity} and the result of Gr\"otschel et al.~\cite{grotschel1981ellipsoid},
it is furthermore possible to construct the probability distributions
corresponding to~$p_U$ and~$p_R$ in polynomial time, using column generation.

\section{Individual fairness} \label{sec:if}
In this section, we further explore the framework for individual fairness.
We link uniform fairness and Rawlsian justice for independence systems.
We additionally discuss implications and novel results for specific graph-theoretical optimization problems.

\subsection{Independence Systems} \label{sec:independence-systems}
The framework that determines $p_U$ and $p_R$ presented in Section~\ref{sec:model} can be applied to
arbitrary set systems~$(A, M)$.
While in general there is no relation between the fairness
quantities~$p_U$ and~$p_R$ other than that~$p_U \le p_R$, we prove here that if~$(A, M)$ forms a so-called
independence system, then~$p_U = p_R$ holds.

Consider a set system~$(A, M)$.
The set system~$M$ is an \emph{independence system} when~$\emptyset \in
M$ and for every~$m \in M$, we have~$m' \in M$ for all~$m' \subseteq m$.
When the set system is an independence system, we call a subset~$m \subseteq A$ \emph{independent} if~$m \in M$.
Independence systems are also known as \emph{simplicial complexes} or \emph{downward-closed set systems}.
Many independence systems are known; we mention here: edge sets that form matchings, edge sets that form forests, edge sets that are the complement of edge covers, vertex sets that are independent sets, vertex sets that are the complement of a vertex cover, and matroids.

In Section~\ref{sec:independence-systems} we show that $p_U=p_R$ for independence systems, and in Sections~\ref{subsec:fair-matching:edges} and \ref{subsec:independent-set} we focus on two special cases: fair matching for edges and fair independent sets.
For these problems, we derive lower and upper bounds on~$p_U$ and~$p_R$.

To show that~$p_U = p_R$ for independence systems, we first prove an
auxiliary lemma in which we use the dual program of~\eqref{eq:lp:pu}, given in~\eqref{eq:lp:pu:dual}.
We show the following lemma, stating that in an optimal solution for
the dual program there is some independent set for which Constraint~\eqref{eq:ip:pdual:m} is binding.
\begin{lemma}
    \label{lem:fixbeta}
    Let~$(A, M)$ be an independence system with~$|A| \ge 2$, and
    let~$(\alpha, \beta)$ be an optimal solution for~\eqref{eq:lp:pu:dual}.
    For each~$\bar a \in A$, there is an independent set~$\bar m \in
    M_{\bar a}$ for which $\sum_{a \in \bar m} \alpha_a + \beta = 0$.
\end{lemma}
\begin{proof}
    Suppose that $\sum_{a \in m} \alpha_a + \beta > 0$ for all~$m \in M_{\bar a}$, for the sake of contradiction.
    Then, there exists an~$\epsilon > 0$ such that
    \[
        \sum_{a \in m} \alpha_a + \beta - \frac{|A|}{|A|-1} \epsilon \ge 0
    \]
    for all~$m \in M_{\bar a}$.
    In particular, we may assume that~$\epsilon \le \beta(|A| - 1)$.
    Define
    \[
        \alpha'_a = \begin{cases}
                        \alpha_a - \epsilon, & \text{if $a = \bar a$,} \\
                        \alpha_a + \frac{\epsilon}{|A| - 1}, & \text{
                            otherwise,}
        \end{cases}
    \]
    i.e., we distribute an $\epsilon$-part of the weight of~$\bar a$
    over all other elements.
    Also, define~$\beta' = \beta - \frac{\epsilon}{|A|-1}$.
    To complete the proof, we will show that~$(\alpha', \beta')$ is a solution
    for~\eqref{eq:lp:pu:dual}.
    Since~$|A| \ge 2$, we have that~$\beta' < \beta$, which contradicts
    the optimality of~$(\alpha, \beta)$.

    First notice that $\sum_{a \in A} \alpha'_{a} = \sum_{a \in A} \alpha_{a} = 1$.
    To show that the solution satisfies the
    Constraints~\eqref{eq:ip:pdual:m}, we consider three cases for~$m$:
    \begin{itemize}
        \item For~$m = \emptyset$, we have, using the upper bound~$\epsilon \le
        \beta(|A| - 1)$,
        \[
            \sum_{a \in m} \alpha'_a + \beta' = \beta - \frac{\epsilon}{|A| - 1} \ge 0.
        \]

        \item For~$m \in M_{\bar a}$, using that $|m| \ge 1$, we have
        \[
            \begin{split}
                \sum_{a \in m} \alpha'_a + \beta'
                &= \sum_{a \in m} \alpha_a + \frac{|m| - 1}{|A| - 1}\epsilon
                - \epsilon + \beta - \frac{\epsilon}{|A| - 1} \\
                &\ge \sum_{a \in m} \alpha_a + \beta
                - \frac{|A|}{|A| - 1}\epsilon
                \ge 0.
            \end{split}
        \]

        \item For~$m \in M\setminus (M_{\bar a} \cup \set{\emptyset})$,
        using that $|m| \ge 1$, we have
        \[
            \begin{split}
                \sum_{a \in m} \alpha'_a + \beta' &= \sum_{a \in m} \alpha_a
                + \beta + \frac{|m| - 1}{|A| - 1} \epsilon \\
                &\ge \sum_{a \in m} \alpha_a + \beta \ge 0.
                \qedhere
            \end{split}
        \]
    \end{itemize}
\end{proof}
We can now show that the two fairness quantities~$p_U$ and~$p_R$ coincide for independence systems.
\begin{theorem} \label{thm:indepsys:pu-pr}
    For any independence system~$(A, M)$, $p_U = p_R$.
\end{theorem}
\begin{proof}
    Observe that $p_U = p_R$ trivially holds if~$|A| = 1$.
    If~$|A| \ge 2$, we consider the dual of~\eqref{eq:lp:pu},
    given in~\eqref{eq:lp:pu:dual}.
    The dual of~\eqref{eq:lp:pr} is similar to~\eqref{eq:lp:pu:dual}, albeit with the additional
    constraints~$\alpha_a \le 0$ for all~$a \in A$.
    To prove the statement, it suffices to show that any optimal
    solution of~\eqref{eq:lp:pu:dual} implicitly satisfies these
    additional constraints.

    For the sake of contradiction, assume that there exists an optimal
    solution~$(\alpha_a, \beta)$ of~\eqref{eq:lp:pu:dual}
    where~$\alpha_{\bar a} > 0$, for some element~$\bar a$.
    By Lemma~\ref{lem:fixbeta}, there exists an independent
    set~$\bar m \in M_{\bar a}$ such
    that~$\beta = -\sum_{a \in \bar m} \alpha_a$.
    However, consider now the independent
    set~$m = \bar m \setminus \set{\bar a}$.
    Note that indeed~$m \in M$ as~$(A, M)$ is an independence system.
    We have
    \[
        \sum_{a \in m} \alpha_a + \beta
        = \sum_{a \in \bar m} \alpha_a - \alpha_{\bar a} + \beta
        = -\alpha_{\bar a} < 0,
    \]
    which violates Constraint~\eqref{eq:ip:pdual:m};
    a contradiction.
\end{proof}

\subsubsection{Individually-Fair Matching for Edges} \label{subsec:fair-matching:edges}
Consider fair matching for edges.
Thus, we consider the set system~$(E, M)$, where~$E$ is the edge set of the graph and~$M$ the set of all matchings.
We assume~$|E| \ge 1$.
We furthermore introduce the notation~$M_e \coloneqq \set{m \in M : e \in m}$ as the set of all matchings that contain the edge~$e \in E$.

Notice that~$(E, M)$ is indeed a independence system, by observing that for any matching~$m$ and edge~$e \in m$, $m \setminus \set{e}$ is also a matching.
By Theorem~\ref{thm:indepsys:pu-pr}, we thus have that~$p_U = p_R$ for fair matching for edges.

To further investigate these two fairness measures in this settings, we turn to the fractional covering number of the hypergraph~$(E, M)$.
The fractional covering number is equal to the \emph{fractional edge coloring number}~$\ecolnumf G$, which can be defined as
\begin{subequations}
    \label{eq:ip:fracedgecol}
    \begin{alignat}{7}
        \ecolnumf G =\quad &&\text{minimize}\quad && \sum_{m \in M} w_m
        &&& \\
        &&\text{subject to}\quad && \sum_{m \in M_e} w_m &\ge 1
        &&\qquad \forall {e \in E}, \\
        &&&& w_m &\ge 0 &&\qquad\forall {m \in M}.
    \end{alignat}
\end{subequations}
By Lemma~\ref{lemma:pr-frac-hypergraph-cov}, we thus have the following full characterization of uniform fairness and Rawlsian justice for fair matchings for edges.
\begin{corollary}
    For any graph~$G$, $p_U = p_R = \frac{1}{\ecolnumf G}$.
\end{corollary}

The proof follows directly from the results in Theorem~\ref{thm:indepsys:pu-pr} and Lemma~\ref{lemma:pr-frac-hypergraph-cov}, using that~$(E, M)$ is an independence system.

Using previous results from (fractional) edge coloring theory, we can derive bounds for~$p_U$ and~$p_R$ that are easy to compute.
A famous result is a theorem of Vizing~\cite{vizing1964estimate} that states the integral edge coloring number~$\ecolnum G$ can only take one of two possible values.
\begin{theorem}[Vizing~\cite{vizing1964estimate}]
    Let~$G$ be a graph.
    Then it holds that~$\maxdeg G \le \ecolnum G \le \maxdeg G + 1$.
\end{theorem}
As shown in~\cite{scheinerman2011fractional}, the same lower and upper bounds hold for the fractional edge coloring number, a theorem we refer to as the fractional Vizing theorem.
\begin{theorem}[Fractional Vizing~\cite{scheinerman2011fractional}]
    For any graph~$G$, it holds that $\maxdeg G \le \ecolnumf G \le \maxdeg G + 1$.
\end{theorem}
This result directly implies that~$p_U$ and~$p_R$ are closely related to the maximum degree of the graph.
\begin{corollary}
    For any graph~$G$, we have that~$p_U$ and~$p_R$ are bounded by~$\frac{1}{\maxdeg G + 1} \le p_U = p_R \le \frac{1}{\maxdeg G}$.
\end{corollary}
We now turn our attention to describing graphs for which the lower and upper bound are attained.
It is shown in~\cite{scheinerman2011fractional} that the lower bound is attained if and only if~$G = K_{2n+1}$, i.e., the complete graph on an odd number of vertices~(with~$n \ge 1$ integer).
We give a sufficient condition for the upper bound to be achieved.

\begin{theorem}
    Let~$G$ be a graph.
    If~$p_U = p_R = \frac{1}{\maxdeg G}$, then it holds that~$\nu(G) \ge \frac{|E|}{\maxdeg G}$.
\end{theorem}
\begin{proof}
    Consider the LP~\eqref{eq:lp:pu} for the independence system~$(E, M)$, and replace the variable~$p$ with the constant~$\frac{1}{\maxdeg G}$.
    Note that this LP has a feasible solution, as~$p_U = p_R = \frac{1}{\maxdeg G}$.
    The corresponding dual is
    \begin{subequations}
    \begin{alignat}{7}
        &&\text{minimize}\quad && \frac{1}{\maxdeg G} \sum_{e \in E} \alpha_e + \beta &&& \\
        &&\text{subject to}\quad && \sum_{e \in m} \alpha_e + \beta &\ge 0
        &&\qquad \forall {m \in M}, \\
        &&&& \alpha_e &\in \reals &&\qquad\forall {e \in E},\\
        &&&& \beta &\in \reals. &&
    \end{alignat}
    \end{subequations}
    Consider now the dual solution with~$\alpha_e = -1$ for every~$e \in E$ and~$\beta = \nu(G)$.
    Note that this solution is indeed feasible, and has dual objective value~$\nu(v) - \frac{|E|}{\maxdeg G}$.
    This expression is negative when~$\nu(v) < \frac{|E|}{\maxdeg G}$.
    By scaling the solution, the dual is unbounded whenever this condition on the size of the maximum matching in the graph holds.
    This implies that the primal LP is infeasible, a contradiction.
    Hence, $\nu(v) \ge \frac{|E|}{\maxdeg G}$.
\end{proof}
In particular, this sufficient condition implies for regular graphs, where~$\frac{|E|}{\maxdeg G} = \frac{|V|}{2}$, that a perfect matching needs to exist for the~$\frac{1}{\maxdeg G}$ upper bound to be achieved for~$p_U$ and~$p_R$.

\subsubsection{Individually-Fair Independent Set} \label{subsec:independent-set}
In this section, we discuss modeling fairness for independent sets in a graph.
Let~$I$ denote the collection of all independent sets in the graph.
Then~$(V, I)$ is an independence system.
Indeed, the empty set is an independent set, and for every independent set~$S \in I$ we have that every subset~$S' \subseteq S$ is also an independent set.
The fractional covering number of the hypergraph~$(V, I)$ coincides with the~\emph{fractional vertex coloring number}~$\chi_f(G)$.
Let~$\pu{is},\pr{is}$ denote uniform fairness and Rawlsian justice for fair independent set, respectively.
We have the following consequence of Theorem~\ref{thm:indepsys:pu-pr}.
\begin{corollary}
    For any graph~$G$, $\pu{is} = \pr{is} = \frac{1}{\chi_f(G)}$.
\end{corollary}
Independent sets are related to many other combinatorial structures in graphs, such as vertex covers.
A set~$C \subseteq V$ is a \emph{vertex cover} if and only if~$V \setminus C$
is an independent set in~$G$.
Because vertex covers arise from the complement of independent sets, the set of all vertex covers~$K$ can be seen as an independence system, with the inclusion relation ``reversed.''
That is, for every~$C \in K$ we also have that~$C' \in K$ for every superset~$C' \supseteq C$.
We call~$(V, K)$ a \emph{reversed independence system}.

Notice that~$V \in K$, and hence a uniform fairness probability of~$1$ can be trivially achieved by only selecting the vertex cover~$V$, the complete vertex set, in the probability distribution.
Therefore, it makes sense to \emph{minimize} the fairness probabilities~$p_U$ and~$p_R$ in the setting of reversed independence systems.
Indeed, this is equivalent to maximizing the probability of not being selected in, e.g., a vertex cover.
This is applicable in situations where it is not desirable to be in the selected vertex cover, for example when being selected results in having to perform a tedious or unwanted task.
Let~$\pu{vc},\pr{vc}$ denote uniform fairness and Rawlsian justice for fair vertex cover, respectively.
Due to the one-to-one correspondence between independent sets and vertex covers via the complement with the vertex set, we thus have the following result.

\begin{theorem}
    For any graph~$G$, $\pu{vc} = \pr{vc} = 1 - \frac{1}{\chi_f(G)}$.
\end{theorem}

We have a similar relation for cliques in the graph.
A subset~$C \subseteq V$ is a \emph{clique} if and only if~$C$ is an independent set in
the complement graph~$\bar G = (V, \bar E)$, where~$\bar E$ is the set of non-edges in~$G$.
Let~$\bar \chi_f = \chi_f(\bar G)$ denote the fractional vertex coloring number of~$\bar G$.
With~$\pu{cl},\pr{cl}$ denoting uniform fairness and Rawlsian justice for fair clique, we obtain the following result.
\begin{theorem}
    For any graph~$G$, $\pu{cl} = \pr{cl} = \frac{1}{\bar \chi_f}$.
\end{theorem}

\subsection{Individually-Fair Matching for Vertices} \label{sec:matching-vertex}
In this section, we focus on fair matching for vertices.
Interestingly, as we describe in the next paragraph, the corresponding set system is \emph{not} an independence system, and hence Theorem~\ref{thm:indepsys:pu-pr} does not apply.
And indeed, instances exist where $p_R > p_U$, see for example the graph in Figure~\ref{fig:punotpr}.

\begin{figure}
    \centering
    \def\ebmfigscale{.65}
    \begin{tikzpicture}[scale=\ebmfigscale,xstep=2,ystep=2]
    \definecolor{baseorange}{HTML}{D55E00}
    \definecolor{baseblue}{HTML}{0072B2}
    \definecolor{basegreen}{HTML}{009E73}
    \definecolor{textcolor}{HTML}{000000}
    \definecolor{tuescarlet}{RGB}{200,24,24}
    \def\radius{1.75mm}
    \def\lw{.4mm}

    \begin{scope}[xscale=2,yscale=2]
        \coordinate (1) at (0.134, .5);
        \coordinate (2) at (0.134, -.5);
        \coordinate (3) at (1, 0);
        \coordinate (4) at (2, 0);
        \coordinate (5) at (3, 0);
    \end{scope}

    \draw[line width=\lw,draw=textcolor] (1) -- (2) node[midway,left] {$a$};
    \draw[line width=\lw,draw=textcolor] (2) -- (3) node[midway,below] {$b$};
    \draw[line width=\lw,draw=textcolor] (1) -- (3) node[midway,above] {$c$};
    \draw[line width=\lw,draw=textcolor] (3) -- (4) node[midway,above] {$d$};
    \draw[line width=\lw,draw=textcolor] (4) -- (5) node[midway,above] {$e$};

    \foreach \v in {1, ..., 5} {
        \fill[fill=textcolor] (\v) circle[radius=\radius];
    }

\end{tikzpicture}
    \caption{A graph with edges~$\set{a, b, c, d, e}$, where for the fair matching for vertices problem, $p_R = \nicefrac 3 4$ and~$p_U = \nicefrac 2 3$. The probability~$p_R$ can be achieved by selecting each of the matchings~$\set{a, e}$, $\set{b, e}$, $\set{c, e}$, $\set{a, d}$ with probability~$\nicefrac 1 4$, whereas~$p_U$ can be achieved by selecting the matchings~$\set{a}$, $\set{b, e}$, $\set{c, e}$ with probability~$\nicefrac 1 3$.}
    \label{fig:punotpr}
\end{figure}
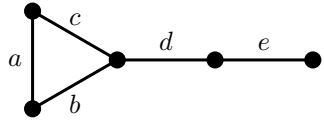

Let~$G=(V, E)$ be a graph and let~$M$ denote the set of all matchings.
For a matching~$m \in M$, let~$V(m) \subseteq V$ denote the set of
vertices covered by the matching~$m$.
Define~$V(M) \coloneqq \set{V(m) : m \in M}$, and let~$M_v \subseteq M$ denote the set of matchings that cover the vertex~$v$.
We can then apply our framework from Section~\ref{sec:model} to the set system~$(V, V(M))$.
Notice however, that this system is not an independence system. Indeed, one cannot remove a single vertex from a subset~$m \in V(M)$ to obtain a different subset in~$V(M)$, as these sets correspond to the endpoints of edges.
We discuss the uniform fairness measure~$p_U$ in Section~\ref{subsec:unifair} and Rawlsian justice~$p_R$ in Section~\ref{subsec:rawlsfair}.

\subsubsection{Uniform Fairness}
\label{subsec:unifair}
We first focus on the uniform fairness measure~$p_U$.
One can immediately observe that if a graph has a perfect matching, then choosing only this perfect matching in the support of the probability distribution~$x$ yields~$p_U = 1$.
The other direction of the statement also holds.

\begin{lemma}
    $G$ has a perfect matching if and only if $p_U = 1$.
\end{lemma}
\begin{proof}
    For the forward direction, we can trivially select the perfect
    matching (with probability~$1$), which covers every vertex.
    For the other direction, assume that~$p_U = 1$ with an optimal
    solution~$x$.
    If we now sample a matching~$m$ according to the probability
    distribution~$x$, then every vertex needs to be in~$m$, and
    hence~$m$ is a perfect matching.
\end{proof}

In general, there also exist graphs for which~$p_U = 0$, for which a characterization is somewhat more involved.
For this, we will use the dual of the LP in~\eqref{eq:lp:pu} applied to the set system~$(V, V(M))$:
\begin{subequations}
    \label{eq:ip:v:pdual}
    \begin{alignat}{7}
        &&\text{minimize}\quad && \beta &&& \\
        &&\text{subject to}\quad && \sum_{v \in V_m} \alpha_v + \beta
        &\ge 0 &&\qquad \forall {m \in M} \\
        &&&& \sum_{v \in V}\alpha_v &= -1, &&\\
        &&&& \alpha_v &\in \reals &&\qquad\forall {v \in V},\\
        &&&& \beta &\in \reals. &&
    \end{alignat}
\end{subequations}
This dual program has~$|M| + 1$ constraints, but for determining whether~$p_U = 0$, it is in fact sufficient to only consider~$|E| + 1$ dual constraints.
This observation is formalized in the following lemma.
\begin{lemma}
    \label{lem:easylife}
    For any graph~$G$, $p_U = 0$ if and only if there
    is an~$\alpha \in \reals^V$ such that~$\alpha_u + \alpha_v \ge 0$ for
    all~$\set{u, v} \in E$ and~$\sum_{v \in V} \alpha_v < 0$.
    Moreover, the set $\set{(x, p) \in \reals_+^E \times \reals_+ : \text{$\sum_{e \in \delta(v)} x_e = p \quad \forall v \in V$ and $\sum_{e \in E} x_e = 1$}}$ is nonempty if and only if~$p_U > 0$.
\end{lemma}
\begin{proof}
    We start with the proof of the first part of the lemma.
    For the forward direction, assume that~$p_U = 0$.
    Then, there exists a dual solution~$(\alpha, \beta)$ with~$\beta = 0$ for~\eqref{eq:ip:v:pdual}.
    Notice that single edges are matchings (of cardinality~$1$), and thus feasibility of the dual solution yields~$\alpha_u + \alpha_v \ge 0$ for
    all~$\set{u, v} \in E$ and~$\sum_{v \in V} \alpha_v = -1 < 0$.

    For the other direction, assume that we have~$\alpha \in \reals^V$ such that~$\alpha_u + \alpha_v \ge 0$ for
    all~$\set{u, v} \in E$ and~$\sum_{v \in V} \alpha_v < 0$.
    Note that, without loss of generality, we may assume that~$\sum_{v \in V} \alpha_v = -1$ by rescaling the solution without violating the constraint for every edge.
    Consider a matching~$m \in M$.
    Notice that the edges in a matching have
    disjoint endpoints, and hence
    \[
        \sum_{v \in V_m} \alpha_v = \sum_{\set{u,v} \in m} [\alpha_u + \alpha_v] \ge 0
    \]
    is trivially satisfied.
    Therefore, $\alpha$ is a feasible solution for the dual in~\eqref{eq:ip:v:pdual} with~$\beta = 0$.
    By weak duality, we have that~${p_U = 0}$.

    For the second part, assume that we have~$p_U > 0$ and that the set~$\set{(x, p) \in \reals_+^E \times \reals_+ : \text{$\sum_{e \in \delta(v)} x_e = p \quad \forall v \in V$ and $\sum_{e \in E} x_e = 1$}}$ is empty.
    Note that we can write this system as~$Ay = b$, $y \ge 0$ where~$y = (x, p)^\top \in \reals_+^E \times \reals_+$ with a suitable constraint matrix~$A$ and vector~$b$.
    Then, by applying Farkas' lemma to this system, there exists a vector~$z = (\alpha, \beta) \in \reals^V \times \reals$ such that~$z^\top b < 0$, $y^\top A \ge 0$.
    Or, equivalently, there exist~$\alpha \in \reals^V$ and~$\beta \in \reals$ such that
    \[
        \beta < 0,\qquad \alpha_u + \alpha_v + \beta \ge 0\quad \forall \set{u,v} \in E,\qquad \sum_{v \in V} \alpha_v \le 0.
    \]
    Eliminating~$\beta$, we can equivalently formulate this as~$\alpha_u + \alpha_v > 0$ for all~$\set{u,v} \in E$ and~$\sum_{v \in V} \alpha_v \le 0$.
    Fix a vertex~$v \in V$ and define for some~$\epsilon > 0$,
    \[
        \alpha'_u = \begin{cases}
                       \alpha_v - \epsilon, & \text{if $u = v$,} \\
                       \alpha_u, & \text{otherwise.}
        \end{cases}
    \]
    Choosing~$\epsilon$ sufficiently small, we have that~$\alpha'_u + \alpha'_v \ge 0$ for all edges~$\set{u,v} \in E$ and~$\sum_{v \in V} \alpha_v < 0$.
    By the first part of the lemma, this implies that~$p_U = 0$, a contradiction.
    Hence, by Farkas' lemma the original system is feasible.
    The other direction is analogous, using the bidirectionality of Farkas' lemma.
\end{proof}
The first part of the above lemma shows that for determining whether~$p_U = 0$, we only
need~$|E| + 1$ dual constraints, instead of constraints for every matching.
We will use this to first characterize when~$p_U = 0$ for the easier class of bipartite graphs.
\begin{lemma}
    If~$G$ is bipartite and contains no perfect matching, then~$p_U = 0$.
\end{lemma}
\begin{proof}
    Let~$U,V$ denote the vertex partition of the two sides of the graph.
    As~$G$ contains no perfect matching, we know via Hall's theorem that
    there exists a subset~$S \subseteq U$ such that~$|N(S)| < |S|$.
    Define now a solution~$\alpha$ to the program~\eqref{eq:ip:v:pdual} as
    \[
        \alpha_v = \begin{cases}
                       1, & \text{if $v \in N(S)$,} \\
                       -1, & \text{if $v \in S$,} \\
                       0, & \text{otherwise.}
        \end{cases}
    \]
    Then $\sum_{v \in V} \alpha_v = |N(S)| - |S| < 0$.
    Consider an edge~$\set{u, v} \in E$.
    Suppose that one of the endpoints is in~$S$, then $\alpha_u + \alpha_v = 1 - 1 = 0$.
    If neither endpoint is in~$S$, then both~$\alpha_u, \alpha_v \ge 0$
    and thus~$\alpha_u + \alpha_v \ge 0$ as well.
    The result now follows by the first part of Lemma~\ref{lem:easylife}.
\end{proof}
The result here might be surprising, as for bipartite graphs~$p_U$ is
either $0$ or $1$, and completely determined by the existence of a perfect matching.

This result generalizes to general graphs in the following manner.
Whether it holds that~$p_U = 0$ depends on the existence of a \emph{fractional perfect
matching} in the graph.
A \emph{fractional matching} is a valuation of the edges~$y_e \in [0, 1]$
such that for every vertex~$v$ we have~$\sum_{e \in \delta(v)} y_e \le 1$.
A \emph{fractional perfect matching} is a fractional matching such
that~$\sum_{e \in \delta(v)} y_e = 1$, or, equivalently,
$\sum_{e \in E} y_e = \frac{|V|}{2}$.

\begin{theorem}
    \label{thm:v:pupositive}
    The following statements are equivalent for any graph~$G$:
    \begin{enumerate}[label=(\roman*)]
        \item $p_U > 0$.
        \item $G$ has a fractional perfect matching.
        \item For every independent set~$S$ in~$G$, we have~$|S| \le |N(
        S)|$.
        \item $G$ has a \emph{Q-factor}: a spanning subgraph where every
        component is regular, with either degree~$1$ on exactly two vertices, or
        degree~$2$ on an odd number of vertices.
        \item There is a partition~$\set{V_1, \dots, V_k}$ of the vertex
        set~$V$ such that, for each~$i$, the graph~$G[V_i]$ is either~$
        K_2$ or a Hamiltonian graph on an odd number of vertices.
        \item $G$ has a
        spanning subgraph where every component is regular with degree
        at least~$1$.
        \item The number of isolated vertices in $G - S$ is at most $|S|$
        for every $S \subseteq V$.
    \end{enumerate}
\end{theorem}
\begin{proof}
    The equivalence of~(iii) and (iv) has been shown by
    Tutte~\cite{tutte19531}.
    Proofs of equivalence with~(ii), (v) and (vii) can be found in the
    book by Scheinerman and Ullman~\cite{scheinerman2011fractional}.
    This leaves (i) and (vi).

    $(\text{iv}) \implies (\text{vi})$: Obviously, if~$G$ has a Q-factor
    then this subgraph is also a regular spanning subgraph with each component of degree at least~$1$.

    $(\text{vi}) \implies (\text{ii})$: Note that every $d$-regular
    graph with $d \ge 1$ has a fractional perfect matching by choosing~$
    y_e = \frac 1 d$ for every edge~$e$.
    Assume that~$G$ has a spanning subgraph where every component is regular of degree at least~$1$, then every component of the spanning subgraph has a
    fractional perfect matching.
    Since the edges and vertices in the components are disjoint,
    combining the fractional perfect matchings yield a fractional
    perfect matching in~$G$ since the subgraph is a spanning subgraph.

    $(\text{i}) \implies (\text{ii})$: Assume that~$p_U > 0$.
    Using the second part of Lemma~\ref{lem:easylife}, there exists~$(x, p) \in \reals_+^E \times \reals_+$ with
    \[
        \sum_{e \in \delta(v)} x_e = p \quad \forall v \in V,\qquad \sum_{e \in E} x_e = 1.
    \]
    Define now~$y_e \coloneqq \frac{|V|}{2} x_e^*$ for all~$e \in E$.
    Then, $y_e \ge 0$ and
    \[
        \sum_{e \in E} y_e = \frac{|V|}{2} \sum_{e\in E} x^*_e = \frac{|V|}{2},
    \]
    and
    \[
        \sum_{e \in \delta(v)} y_e = \frac{|V|}{2} p^*
    \]
    for all~$v \in V$.
    Summing these over all vertices, we obtain via double counting
    that~$2 \sum_{e \in E} y_e = \frac{|V|^2}{2} p^*$ and for this reason~$\frac{|
    V|^2}{4} p^* = \frac{|V|}{2}$.
    But then~$\sum_{e \in \delta(v)} y_e = \frac{|V|}{2} p^* = 1$, and
    thus~$y$ is a fractional perfect matching.

    $(\text{ii}) \implies (\text{i})$: this is analogous with the
    previous case, but with the transformation~$x_e \coloneqq
    \frac{2}{|V|} y^*_e$ for every edge~$e \in E$, where~$y^*$ is a
    fractional perfect matching.
\end{proof}

This implies the following, somewhat surprising, lower bound on $p_U$
when it is positive.
\begin{corollary}
  \label{cor:matchingVertex}
    If $p_U > 0$, then $p_U \ge \frac 2 3$.
\end{corollary}
\begin{proof}
    If~$p_U > 0$ then $G$ has a Q-factor, which is a collection of
    vertex-disjoint separate edges and odd cycles.
    Note that in an odd cycle on $n$~vertices we can achieve~$p_U = \frac{n-1}{n}$ by giving every maximum matching weight~$\frac 1 n$, which is
    smallest for $n = 3$.
    For separate edges, we have $p_U = 1$.
    Thus, combining the probabilities for the separate components, we
    are always able to achieve at least~$p_U \ge \frac 2 3$ for the whole graph.
\end{proof}

\subsubsection{Rawlsian Justice}
\label{subsec:rawlsfair}
In the previous section, we have seen that requiring uniform fairness may cause $p_U$ to become zero.
When~$p_U = 0$, we might hope that the more relaxed notion of Rawlsian justice still yields a desirable solution.
To this end, we relate~$p_R$ to known graph invariants.

In the case of fair matchings for vertices, the covering number of the corresponding hypergraph (Section~\ref{sec:hypergraph}) corresponds to the \emph{matching covering number} of~$G$, as introduced in~\cite{wang2014matching,ferhat2022many}.
We define its fractional analogue, the \emph{fractional matching covering number}~$\mcf G$, with the following LP\@.
\begin{subequations}
    \label{eq:ip:v:fracmc}
    \begin{alignat}{7}
        \mcf G
        =\quad &&\text{minimize}\quad && \sum_{m \in M} w_m &&& \\
        &&\text{subject to}\quad && \sum_{m \in M_v} w_m &\ge 1
        &&\qquad \forall {v \in V}, \\
        &&&& w_m &\ge 0 &&\qquad\forall {m \in M}.
    \end{alignat}
\end{subequations}
By Lemma~\ref{lemma:pr-frac-hypergraph-cov}, we have~$p_R = \frac{1}{\mcf G}$.
Using previous results for the integral matching covering number, we can derive a lower bound on~$p_R$.
\begin{lemma}
    For any regular graph~$G$, $p_R \ge \frac 2 3$.
    If~$G$ is non-regular, it holds that $p_R \ge \frac{1}{\maxdeg G - \mindeg G + 1}$.
\end{lemma}
\begin{proof}
    If~$G$ is regular, then~$p_U > 0$ by~(vi) of Theorem~\ref{thm:v:pupositive} (the graph itself is a regular spanning subgraph).
    Then by Corollary~\ref{cor:matchingVertex}, we have~$p_R \ge p_U \ge \frac 2 3$.
    Suppose now that~$G$ is not regular.
    From~\cite{wang2014matching}, we have the following upper bound on
    the integral matching covering number
    \[
        \mc G \le \max\set{2, \maxdeg G - \mindeg G + 1}.
    \]
    Since we have~$\maxdeg G - \mindeg G \ge 1$, the upper bound simplifies to~$\mc G \le
    \maxdeg G - \mindeg G + 1$.
    This gives
    \[
        p_R = \frac{1}{\mcf G} \ge \frac{1}{\mc G}
        \ge \frac{1}{\maxdeg G - \mindeg G + 1}. \qedhere
    \]
\end{proof}
Note that this lower bound on~$p_R$ is tight, in the sense that we have
equality when~$G$ is regular or a star graph with at least three vertices.
For graphs that are ``almost regular'', in the sense that the maximum degree is close to the minimum degree, we can still achieve an acceptabele lower
bound, even for graphs where~$p_U$ is zero.

\section{Group Fairness} \label{sec:gf}
In the previous sections, we have considered fairness with respect to \emph{individual} elements in the ground set of the considered set system.
In this section, we formalize group fairness in our setting of set systems in Section~\ref{subsec:gfc}, and furthermore analyze the complexity of finding group-fair solutions for the fair matching for vertices problem in Section~\ref{subsec:gf-matching-v}.

\subsection{Group fairness constraints} \label{subsec:gfc}
Our problem input now not only consists of a set system~$(A, M)$, but also of a collection of \emph{groups}~$\mathcal G = \set{G_1, \dots, G_k}$, with $G_1, \dots, G_k \subseteq A$ subsets of the ground set~$A$.
We assume that the groups are pairwise disjoint;
when considering groups that arise from a single sensitive attribute, this is in many cases a natural assumption.
We impose group fairness constraints on the possible solutions, i.e., we require \emph{ex-post} group fairness, which ensures that the group fairness constraints hold in any solution we consider.
We thus restrict the set~$M$ to only include solutions that satisfy the group fairness constraints, $M_\text{gf} \subseteq M$.

Note that we considered \emph{ex-ante} individual fairness in Section~\ref{sec:if}, while now requiring \emph{ex-post} group fairness.
Although it is possible to also handle group fairness in an ex-ante setting, group fairness constraints are often imposed on a problem to ensure representation of protected groups.
Usually, these are hard conditions imposed on solutions, and in some cases required by regulations or law.
In these cases, ensuring ex-ante fairness is not sufficient, and the fairness constraints must be enforced ex-post.
Furthermore, it is possible to combine ex-post group fairness with ex-ante individual fairness in our framework in Section~\ref{sec:if}.
Indeed, instead of optimizing over the set system~$(A, M)$, we can instead take the system~$(A, M_\text{gf})$.
Note that, in general, the restricted system is not an independence system, even when~$(A, M)$ is.
In particular, when an absolute group fairness constraint specifies a lower bound (not trivially equal to zero) on the number of elements of the group, the independence system property does not hold.

We consider two types of group fairness constraints.
For every group~$G_i \in \mathcal G$, we define an \emph{absolute group fairness constraint} by defining a lower bound~$\ell_i \in \naturals$ and an upper bound~$u_i \in \naturals$ on the number of individuals from group~$G_i$ that should appear in a solution.
Then, the set~$M_\text{gf}^\text{abs} \subseteq M$ of solutions that satisfy these group fairness constraints is given by
\begin{equation}
    M_\text{gf}^\text{abs} = \{m \in M : \text{$\ell_i \le |m \cap G_i| \le u_i$ for all $i \in [k]$} \}.
\end{equation}

We also allow \emph{relative group fairness constraints}, which we define for a pair of groups~$G_i, G_j \in \mathcal G$ by introducing a number~$\alpha_{ij} \in \reals_+$ that defines the relative ratio between the cardinality of elements in the solution from group~$G_i$ and~$G_j$.
Relative group fairness constraints may not necessarily be present for all pairs of groups.
Therefore, we define for a set~$S \subseteq [k]^2$ of pairs of groups, the set~$M_\text{gf}^\text{rel}(S) \subseteq M$ of solutions that satisfy relative group fairness constraints by
\begin{equation}
    M_\text{gf}^\text{rel}(S) = \{m \in M : \text{$|m \cap G_i| \le \alpha_{ij} |m \cap G_j|$ for all $(i,j) \in S$} \}.
\end{equation}
The set of group-fair solutions is then given by~$M_\text{gf} = M_\text{gf}^\text{abs} \cap M_\text{gf}^\text{rel}(S)$, for~$S \in [k]^2$.

In the remainder, we analyze the complexity of determining group-fair solutions for the group-fair matching for edges and group-fair matching for vertices problems.
For fair matching for edges, notice that the restricted case of two groups~$G_1, G_2 \subseteq E$ and the relative group fairness constraint that for a matching~$m$, $|m \cap G_1| = |m \cap G_2|$, i.e., the number of matching edges from both groups must be equal, is equivalent with the well-known \textsc{Exact-Matching} problem~\cite{exactmatching}.
The computational complexity of \textsc{Exact-Matching} is a long-standing open problem.

\subsection{Group-Fair Matching for Vertices} \label{subsec:gf-matching-v}
In this section, we show that finding group-fair matchings for vertices can be done in polynomial time, for any kind of absolute and/or relative group fairness constraints we consider.
We show this by first showing that a related problem can be solved in polynomial time.

Let~$G = (V, E)$ be a given graph, where the vertices are weighted by~$w\colon V \to \reals$ and colored with colors~$[k]$.
The coloring need not be proper, i.e., two vertices of the same color may be adjacent.
Let~$r_i \in \naturals$ for every~$i \in [k]$.
The \textsc{Exact-Budgeted Matching} problem asks for a matching covering exactly~$r_i$ vertices of every color~$i \in [k]$, of maximum weight of the vertices covered by the matching.

\begin{theorem}
    The \textsc{Exact-Budgeted Matching} problem can be solved in polynomial time.
\end{theorem}
\begin{proof}
    Let~$G = (V, E)$, with vertices colored with colors~$[k]$ and~$r_i \in \naturals$ for every color~$i \in [k]$ be an instance of the \textsc{Exact-Budgeted Matching} problem.
    Let~$w\colon V \to \reals$.

    Let~$G_i$ denote the set of vertices colored with~$i$, and let~$n_i = |G_i|$, for every~$i \in [k]$.
    We construct an auxiliary graph~$G'=(V', E')$ in the following way.
    The set of vertices~$V'$ consists of all vertices~$V$ and~$n_i - r_i$ additional vertices~$V^i$ for every color~$i \in [k]$.
    Thus, $V' = V \cup \left( \bigcup_{i \in [k]} V^i \right)$.
    The set of edges~$E'$ consists of all edges~$E$ and the additional edges~$\set{ \set{u, v} : u \in V^i,\ v \in G_i }$ for every~$i \in [k]$.
    We furthermore define weights~$w'\colon E \to \reals$ on the edges, as~$w'(\set{u, v}) = w(u) + w(v)$ if~$\set{u, v} \in E$, and~$w'(\set{u, v}) = 0$ otherwise.
    See Figure~\ref{fig:auxgraph} for an example construction.

    We now argue that every solution to \textsc{Exact-Budgeted Matching} corresponds to a maximum-weight perfect matching in~$G'$ with weights~$w'$.
    Let~$m$ be a solution of \textsc{Exact-Budgeted Matching}.
    Construct a matching~$m'$ in~$G'$ by taking~$m$ and including $n_i - r_i$ matching edges between~$G_i$ and~$V^i$, for every color~$i \in [k]$.
    Note that this is always possible, since~$m$ covers exactly~$r_i$ vertices of color~$i$.
    Note that~$m'$ is a maximum-weight perfect matching in~$G'$.

    For the other direction, let~$m'$ be a maximum-weight perfect matching in~$G'$.
    Let now~$m = m' \cap E$, which exactly is solution to~\textsc{Exact-Budgeted Matching}.

    The size of~$G'$ is polynomially bounded in the size of the original instance.
    Since finding a perfect matching in~$G'$ can be done in polynomial time, we can solve \textsc{Exact-Budgeted Matching} in polynomial time.
\end{proof}

\begin{figure}
    \centering
    \def\ebmfigscale{.65}
    \begin{tikzpicture}[scale=\ebmfigscale]
    \definecolor{baseorange}{HTML}{D55E00}
    \definecolor{baseblue}{HTML}{0072B2}
    \definecolor{basegreen}{HTML}{009E73}
    \definecolor{textcolor}{HTML}{595959}
    \definecolor{tuescarlet}{RGB}{200,24,24}
    \def\radius{1.75mm}
    \def\lw{.4mm}

    \coordinate (l1) at (1, 1);
    \coordinate (l2) at (1, 2);
    \coordinate (l3) at (1, 3);

    \coordinate (t1) at (2, 4);
    \coordinate (t2) at (3, 4);
    \coordinate (t3) at (4, 4);

    \coordinate (r1) at (5, 2);
    \coordinate (r2) at (5, 3);

    \coordinate (nl1) at (0, 1.5);
    \coordinate (nl2) at (0, 2.5);

    \coordinate (nt1) at (3, 5);

    \coordinate (nr1) at (6, 2.5);

    \draw[line width=\lw,draw=textcolor] (t1) -- (r1);
    \draw[line width=\lw,draw=textcolor] (t1) -- (l2);
    \draw[line width=\lw,draw=textcolor] (l1) -- (r2);
    \draw[line width=\lw,draw=textcolor] (l3) -- (r1);
    \draw[line width=\lw,draw=textcolor] (l2) -- (r2);

    \draw[line width=\lw,draw=textcolor] (t3) -- (r1);
    \draw[line width=\lw,draw=textcolor] (t2) -- (l2);

    \foreach \i in {1, ..., 3} {
        \foreach \j in {1, ..., 2} {
            \draw[line width=\lw,draw=textcolor!50] (l\i) -- (nl\j);
        }
    }

    \foreach \i in {1, ..., 3} {
        \foreach \j in {1, ..., 1} {
            \draw[line width=\lw,draw=textcolor!50] (t\i) -- (nt\j);
        }
    }

    \foreach \i in {1, ..., 2} {
        \foreach \j in {1, ..., 1} {
            \draw[line width=\lw,draw=textcolor!50] (r\i) -- (nr\j);
        }
    }

    \fill[fill=baseblue] (l1) circle[radius=\radius];
    \fill[fill=baseblue] (l2) circle[radius=\radius];
    \fill[fill=baseblue] (l3) circle[radius=\radius];

    \fill[fill=baseorange] (t1) circle[radius=\radius];
    \fill[fill=baseorange] (t2) circle[radius=\radius];
    \fill[fill=baseorange] (t3) circle[radius=\radius];

    \fill[fill=basegreen] (r1) circle[radius=\radius];
    \fill[fill=basegreen] (r2) circle[radius=\radius];

    \fill[fill=baseblue!50] (nl1) circle[radius=\radius];
    \fill[fill=baseblue!50] (nl2) circle[radius=\radius];

    \fill[fill=baseorange!50] (nt1) circle[radius=\radius];

    \fill[fill=basegreen!50] (nr1) circle[radius=\radius];

\end{tikzpicture}
    \hspace{1cm}
    \begin{tikzpicture}[scale=\ebmfigscale]
    \definecolor{baseorange}{HTML}{D55E00}
    \definecolor{baseblue}{HTML}{0072B2}
    \definecolor{basegreen}{HTML}{009E73}
    \definecolor{textcolor}{HTML}{595959}
    \definecolor{tuescarlet}{RGB}{200,24,24}
    \def\radius{1.75mm}
    \def\lw{.4mm}

    \coordinate (l1) at (1, 1);
    \coordinate (l2) at (1, 2);
    \coordinate (l3) at (1, 3);

    \coordinate (t1) at (2, 4);
    \coordinate (t2) at (3, 4);
    \coordinate (t3) at (4, 4);

    \coordinate (r1) at (5, 2);
    \coordinate (r2) at (5, 3);

    \coordinate (nl1) at (0, 1.5);
    \coordinate (nl2) at (0, 2.5);

    \coordinate (nt1) at (3, 5);

    \coordinate (nr1) at (6, 2.5);

    \draw[line width=\lw,draw=textcolor!25] (t1) -- (r1);
    \draw[line width=\lw,draw=textcolor!25] (t1) -- (l2);
    \draw[line width=\lw,draw=textcolor!25] (l1) -- (r2);
    \draw[line width=\lw,draw=textcolor!25] (l3) -- (r1);
    \draw[line width=\lw,draw=textcolor!25] (l2) -- (r2);

    \draw[line width=\lw,draw=textcolor!25] (t3) -- (r1);
    \draw[line width=\lw,draw=textcolor!25] (t2) -- (l2);

    \draw[line width=\lw,draw=tuescarlet] (t3) -- (r1);
    \draw[line width=\lw,draw=tuescarlet] (t2) -- (l2);

    \foreach \i in {1, ..., 3} {
        \foreach \j in {1, ..., 2} {
            \draw[line width=\lw,draw=textcolor!25] (l\i) -- (nl\j);
        }
    }

    \foreach \i in {1, ..., 3} {
        \foreach \j in {1, ..., 1} {
            \draw[line width=\lw,draw=textcolor!25] (t\i) -- (nt\j);
        }
    }

    \foreach \i in {1, ..., 2} {
        \foreach \j in {1, ..., 1} {
            \draw[line width=\lw,draw=textcolor!25] (r\i) -- (nr\j);
        }
    }

    \draw[line width=\lw,draw=tuescarlet] (l3) -- (nl2);
    \draw[line width=\lw,draw=tuescarlet] (l1) -- (nl1);
    \draw[line width=\lw,draw=tuescarlet] (t1) -- (nt1);
    \draw[line width=\lw,draw=tuescarlet] (r2) -- (nr1);

    \fill[fill=baseblue] (l1) circle[radius=\radius];
    \fill[fill=baseblue] (l2) circle[radius=\radius];
    \fill[fill=baseblue] (l3) circle[radius=\radius];

    \fill[fill=baseorange] (t1) circle[radius=\radius];
    \fill[fill=baseorange] (t2) circle[radius=\radius];
    \fill[fill=baseorange] (t3) circle[radius=\radius];

    \fill[fill=basegreen] (r1) circle[radius=\radius];
    \fill[fill=basegreen] (r2) circle[radius=\radius];

    \fill[fill=baseblue!50] (nl1) circle[radius=\radius];
    \fill[fill=baseblue!50] (nl2) circle[radius=\radius];

    \fill[fill=baseorange!50] (nt1) circle[radius=\radius];

    \fill[fill=basegreen!50] (nr1) circle[radius=\radius];

\end{tikzpicture}
    \caption{An auxiliary graph for \textsc{Exact-Budgeted Matching}, with the three colors blue (left), orange (top), and green (right) and color requirements 1, 2, and 1, respectively. A perfect matching exists in this graph (displayed on the right), and corresponds to a solution to \textsc{Exact-Budgeted Matching}.}
    \label{fig:auxgraph}
\end{figure}
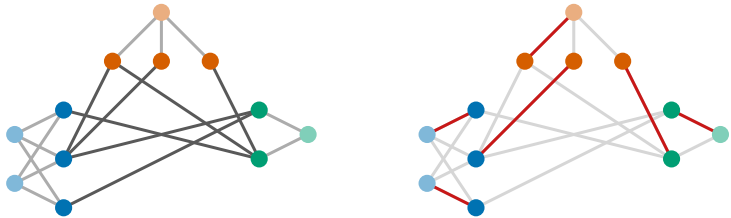

We can use the \textsc{Exact-Budgeted Matching} problem to find group-fair matchings for vertices in the following way.
Any absolute or relative group fairness constraint can alternatively be represented by a set of exact requirements $(r_1, \dots, r_k) \in \naturals^k$ on the number of individuals per group that are covered by the matching, by enumerating all possible values.
Since the requirements~$r_i$ for a group~$i$ are bounded by the number of vertices~$n$, the set of exact requirements is bounded by~$n^k$.
Every exact requirement corresponds to an instance of \textsc{Exact-Budgeted Matching}, where the groups are the colors of the vertices in the graph, which we can solve in polynomial time.
For a constant~$k$ number of groups, we can thus optimize over group-fair matchings in polynomial time, as summarized by the following corollary.
The assumption that the number of groups~$k$ is constant is often natural.
In many cases, groups are defined by only a single or small number of predetermined attributes, such as blood type, in the example of kidney exchange programs.

\begin{corollary}
    Let~$G = (V, E)$ be a graph with a constant number of pairwise disjoint groups~$G_1, \dots, G_k \subseteq V$ of vertices and vertex weights~$w\colon V \to \reals$.
    Let~$M_\mathrm{gf}$ denote the set of matchings restricted to absolute and/or relative group fairness constraints.
    Then finding a matching in $M_\mathrm{gf}$ with maximum total vertex-weight can be done in polynomial time.
\end{corollary}

Furthermore, this means that combining group fairness and individual fairness with our framework in Section~\ref{sec:model}, we can find~$p_U$ and~$p_R$ for ex-post group-fair matchings in polynomial time.

\section{Conclusion} \label{sec:conclusion}
In this work, we analyzed in-depth two individual fairness measures, uniform fairness and Rawlsian justice, in the setting of graph-theoretical optimization problems.
Linking these fairness measures to fractional graph and hypergraph theory, yielding new insights for the general class of problems referred to as independence systems, and for specific problems such as fair matching and fair independent set.
We have furthermore analyzed how to combine ex-ante individual fairness with ex-post group fairness, and analyzed the complexity in more detail for fair matching.
In further directions of research, we are interested how our theoretical approach can be extended to more advanced individual fairness measures, such as minimizing the deviation from the average probability of being included in a matching or other measures that cannot be represented by a linear model.
Additionally, we are interested in exploring the trade-off between fairness and the social optimum from a theoretical standpoint.
In particular, characterizing instances where a certain value for~$p_R$ and~$p_U$ can be guaranteed while restricting to only maximum-size sets in the set system, and classifying the differences between the Pareto fronts for uniform fairness and Rawlsian justice from a theoretical perspective.

\section*{Acknowledgements}
This research is supported by NWO Gravitation Project NETWORKS, Grant Number 024.002.003.

\bibliographystyle{elsarticle-num}
\bibliography{references}

\end{document}